\documentclass[12pt]{article}
\usepackage{amsmath, amsfonts, amsthm, amssymb}
\usepackage{graphicx}
\usepackage{float}
\usepackage{verbatim}% for comment
\usepackage{pdfsync}

\numberwithin{equation}{section}
\newtheorem{thm}{Theorem}[section]
\newtheorem{lma}[thm]{Lemma}

\newtheorem{defn}[thm]{Definition}

\newtheorem{prop}[thm]{Proposition}

\renewcommand{\geq}{\geqslant}
\renewcommand{\leq}{\leqslant}
\renewcommand{\H}{\text{H}}

\title{The horizon problem for prevalent surfaces}
\author{K. J. Falconer and J. M. Fraser}

\begin{document}
\maketitle

\begin{abstract}
We investigate the box dimensions of the horizon of a fractal surface defined by a function $f \in C[0,1]^2 $.  In particular we show that a prevalent surface satisfies the `horizon property', namely that the box dimension of the horizon is one less than that of the surface.  Since a prevalent surface has box dimension 3, this does not give us any information about the horizon of surfaces of dimension strictly less than 3.  To examine this situation we introduce spaces of functions with surfaces of upper box dimension at most $\alpha$, for $\alpha \in [2,3)$.  In this setting the behaviour of the horizon is more subtle.  We construct a prevalent subset of these spaces where the lower box dimension of the horizon lies between the dimension of the surface minus one and 2.  We show that in the sense of prevalence these bounds are as tight as possible if the spaces are defined purely in terms of dimension. However, if we work in Lipschitz spaces, the horizon property does indeed hold for prevalent functions. Along the way, we obtain a range of properties of box dimensions of sums of functions.
\end{abstract}

\section{Introduction and main results}

In this section we introduce the horizon problem, that is the problem of relating the dimension of the horizon of a fractal surface to the dimension of the surface itself. Our main results, which are of a generic nature, depend on the notion of prevalence. 

\subsection{The horizon problem}

For $d \in \mathbb{N}$ let
\[
C[0,1]^d = \{f:[0,1]^d \to \mathbb{R}\, \big\vert \text{ $f$ is continuous} \}.
\]
The {\it graph} of a function $f \in C[0,1]^d$ is the set $G_f=\{(\mathbf{x},f(\mathbf{x})):\mathbf{x} \in [0,1]^d \} \subset [0,1]^d\times \mathbb{R}$.  We shall refer to $G_f$ as a {\it curve} when $d=1$ and as a {\it surface} when $d=2$.
\begin{defn}
Let $f \in C[0,1]^2$.  The horizon function, $H(f) \in C[0,1]$, of $f$ is defined by
\[
H(f)(x) = \sup_{y \in [0,1]} f(x,y).
\]
\end{defn}
We are interested in the relationship between the dimension of the graph of a fractal surface and the dimension of the graph of its horizon.  A `rule of thumb' is that the dimension of the horizon should be one less than the dimension of the surface.  When this is the case we will say that the surface satisfies the `horizon property'.  However, the horizon property is certainly not true in general.  Consider, for example, a surface which is very smooth except for one small region at the bottom of a depression where it has dimension 3.  This irregularity would not affect the horizon which would simply have dimension 1. Thus we can say little about the relationship between the dimensions of the surface and  its horizon for \emph{all} surfaces. Nevertheless, one can consider the `generic' situation or alternatively  one can restrict attention to specific classes of fractal surfaces. 
\\ \\
In \cite{randomhorizon,brownianhorizon} potential theoretic methods were used to find bounds for the Hausdorff dimension for the horizon of index-$\alpha$ Brownian fields.  In particular the index-$\tfrac{1}{2}$ Brownian surfaces almost surely satisfies the horizon property for Hausdorff dimension.
\\ \\
Here we consider the horizon problem for box dimension. We will say that $f \in C[0,1]^2$ satisfies the {\it horizon property (for box dimension)} if the box dimensions of $G_f$ and $G_{H(f)}$ exist and
\[
\dim_\text{B} G_{H(f)} = \dim_\text{B} G_{f} -1.
\]
We examine the horizon problem for a generic surface; of course, there are many ways of defining `generic', but since $C[0,1]^2$ is an infinite dimensional vector space it is natural to appeal to the notion of `prevalence'.

%longer version
\begin{comment}
It is natural to investigate the relationship between the dimension of a surface and the dimension of its horizon.  In \cite{randomhorizon} potential theory is used to estimate the Hausdorff dimension of random surfaces.  In particular, upper and lower bounds are given for the Hausdorff dimension for the horizon of index-$\alpha$ Brownian fields.  Brownian surfaces were also studied more recently in \cite{brownianhorizon}.  The Authors provide partial improvements to the estimates in \cite{randomhorizon} and, in particular, they show that the horizon of an index-$\tfrac{1}{2}$ Brownian surface almost surely has the same H\"older exponent as the original surface.  As a result of this, index-$\tfrac{1}{2}$ Brownian surfaces almost surely satisfy the horizon property for Hausdorff dimension.  The authors also provide numerical results suggesting that this result remains true for arbitrary index.
\end{comment}

\subsection{Prevalence}

`Prevalence' provides one way of describing the generic behavior of a class of mathematical objects.  In a finite dimensional vector space Lebesgue measure provides a natural tool for deciding if a property is generic.  Namely, if the set of elements without the property is a Lebesgue null set then the property is `generic' from a measure theoretical point of view.  However, when the space in question is infinite dimensional this approach breaks down because there is no useful analogue of Lebesgue measure in the infinite dimensional setting.  The theory of prevalence has been developed to address this situation, see the excellent survey papers \cite{prevalence1,prevalence}.
We  give a brief reminder of the key definitions.

\begin{defn}
A \emph{completely metrizable topological vector space} is a vector space $X$ on which there exists a metric $d$ such that $(X,d)$ is complete and such that the vector space operations are continuous with respect to the topology induced by $d$.
\end{defn}

Some sources also require that every point in $X$ is closed in the topology induced by $d$.  This will be trivially true in all of our examples and so we omit it, see \cite{rudin} for more details.
Note that a complete normed space is a completely metrizable topological vector space with the topology induced by the norm.

\begin{defn}
Let $X$ be a completely metrizable topological vector space.  A set $F \subseteq X$ is \emph{prevalent} if the following conditions are satisfied.
\\ 
1)\, \,$F$ is a Borel set;\\
2)  There exists a Borel measure $\mu$ on $X$ and a compact set $K \subseteq X$ such that $0<\mu(K) < \infty$ and
\[
\mu\big(X \setminus (F+x)\big) = 0
\]
for all $x \in X$.
\\
The complement of a prevalent set is called a \emph{shy} set.
\end{defn}

Notice that we can assume that $\mu$ is supported by $K$ in the above definition, otherwise we could replace $\mu$ with the measure $\mu \vert_K$ which would still satisfy condition (2).
\\ \\
Since prevalence was introduced as an analogue of `Lebesgue-almost all' for infinite dimensional spaces it is perhaps not surprising that the measure $\mu$ mentioned in the above definition is often Lebesgue measure concentrated on a finite dimensional subset of $X$.

\begin{defn}
A $k$-dimensional subspace $P \subseteq X$ is called a \emph{probe} for a Borel set $F\subseteq X$ if
\[
\mathcal{L}_P \big(X \setminus (F+x)\big) = 0
\]
for all $x \in X$ where $\mathcal{L}_P$ denotes $k$-dimensional Lebesgue measure on $P$ in the natural way.  We call $F$ $k$-\emph{prevalent} if it admits a $k$-dimensional probe.
\end{defn}

The existence of a probe is clearly a sufficient condition for a set $F$ to be prevalent.

\subsection{Main results}

We will be concerned with spaces of functions defined by the box dimensions of their graphs.
Recall that the {\it lower} and {\it upper box dimensions} (or {\it box-counting dimensions}) of a bounded subset $F$ of $\mathbb{R}^d$ are given by
\begin{equation}\label{lbox}
\underline{\dim}_\text{B} F = \underline{\lim}_{\delta \to 0} \frac{\log N_\delta (F)}{-\log \delta}
\end{equation}
and
\begin{equation}\label{ubox}
\overline{\dim}_\text{B} F = \overline{\lim}_{\delta \to 0} \frac{\log N_\delta (F)}{-\log \delta}
\end{equation}
respectively, where $N_\delta (F)$ is the number of cubes in a $\delta$-mesh which intersect $F$.  If $\underline{\dim}_\text{B} F = \overline{\dim}_\text{B} F$ then we call the common value the {\it box dimension} of $F$ and denote it by $\dim_\text{B} F$.  For basic properties of box dimension see \cite{falconer}.
\\ \\
Let $d\in \mathbb{N}$  and  $\alpha \in [d,d+1]$, and define
\[
C_\alpha[0,1]^d= \{f \in C[0,1]^d : \overline{\dim}_\text{B} G_f \leq \alpha \}
\]
and
\[
D_\alpha[0,1]^d = \{f \in C_\alpha[0,1]^d : \underline{\dim}_\text{B} G_f =\overline{\dim}_\text{B} G_f = \alpha \}.
\]
There is a natural complete metric $d_{\alpha,d}$ on $C_\alpha[0,1]^d$ which we will construct in Section 3. We write $d_\infty$ to denote the metric on  $C_\alpha[0,1]^d$ defined by the norm $\| \cdot \|_\infty$.
\\ \\
The following result, that a prevalent surface has upper and lower box dimension as big as possible, is included to put our results on horizons into context.

\begin{thm}  \label{01}\hspace{1mm}
\begin{itemize}
\item[(1)] $D_{d+1}[0,1]^d$ is a 1-prevalent subset of $(C[0,1]^d,d_\infty)$;

\item[(2)] For $\alpha \in [d,d+1)$ the set $D_\alpha[0,1]^d$ is a 1-prevalent subset of $(C_\alpha[0,1]^d, d_{\alpha,d})$.
\end{itemize}
\end{thm}

%Theorem \ref{01} can be obtained in a similar, but simpler, way to Theorem \ref%{main}, so we omit the proof here. 
Indeed, it  was shown in \cite{mcclure} that the graph of a prevalent function in $(C[0,1],d_\infty)$ has upper box dimension 2. Also, Theorem \ref{01} (1) was very recently obtained in \cite{lowerprevalent}, and a slight weakening of Theorem \ref{01} (1) (with `1-prevalent' replaced just by `prevalent') was given in \cite{shaw} using a completely different method without a probe. 
\\ \\
We will present our results on horizons for surfaces $G_f $ where $f \in C[0,1]^2$ though they may be extended without difficulty to `horizons' of higher dimensional graphs. Our main result is in two parts. Firstly, a prevalent surface satisfies the horizon property.  Specifically, in Theorem \ref{main} (1), we show that a prevalent surface, which according to Theorem \ref{01} (1) has box dimension 3, has  a horizon with box dimension 2.  However, this does not give us any information about the horizon dimensions of surfaces with box dimension strictly less than 3.  Thus in the second part, Theorem \ref{main} (2), we give bounds on the box dimension of the horizon of a prevalent surface in $C_\alpha[0,1]^2$. 
To formulate this, we let
\[
F_\alpha[0,1]^2 = \{f \in C_\alpha[0,1]^2 : \dim_\text{B} G_f = \alpha \text{ and } \alpha-1 \leq \underline{\dim}_\text{B} G_{H(f)} \leq  \overline{\dim}_\text{B} G_{H(f)} \leq 2\}.
\]
Thus $F_\alpha[0,1]^2$ is the set of functions in $C_\alpha[0,1]^2$ for which the box dimension exists and is as big as possible and for which the upper and lower box dimension of the horizon are bounded below by the box dimension of the original surface minus 1.  Note that taking $\alpha=3$ the box dimension of the horizon exists for all $f \in F_3 [0,1]^2$ and is equal to the box dimension of the original surface minus 1.

\begin{thm}  \label{main}\hspace{1mm}
\begin{itemize}
\item[(1)] $F_{3}[0,1]^2$ is a 1-prevalent subset of $(C[0,1]^2,d_\infty)$;

\item[(2)] For $\alpha \in [2,3)$ the set $F_\alpha[0,1]^2$ is a 1-prevalent subset of $(C_\alpha[0,1]^2, d_{\alpha,2})$.
\end{itemize}
\end{thm}
For $\alpha<3$ we do not have precise bounds on the box dimension of the horizon of a prevalent surface.  However, the following theorem shows that our bounds are as tight as possible.

\begin{thm} \label{tight}
Let $\alpha \in [2,3)$ and let  $U_\alpha[0,1]^2$ and $L_\alpha[0,1]^2$ be defined by
\[
U_\alpha[0,1]^2 = \{f \in C_\alpha[0,1]^2 : \dim_\text{\emph{B}} G_f = \alpha \text{ and } \alpha-1 \leq \underline{\dim}_\text{\emph{B}} G_{H(f)} \leq \overline{\dim}_\text{\emph{B}} G_{H(f)} < 2\}
\]
and
\[
L_\alpha[0,1]^2 = \{f \in C_\alpha[0,1]^2 : \dim_\text{\emph{B}} G_f = \alpha \text{ and } \alpha-1 < \underline{\dim}_\text{\emph{B}} G_{H(f)} \leq \overline{\dim}_\text{\emph{B}} G_{H(f)} \leq 2\}.
\]
Then
\begin{itemize}
\item[(1)] $U_\alpha[0,1]^2$ is not a prevalent subset of $(C_\alpha[0,1]^2, d_{\alpha,2})$;

\item[(2)] $L_\alpha[0,1]^2$ is not a prevalent subset of $(C_\alpha[0,1]^2,d_{\alpha,2})$.

\end{itemize}

\end{thm}

Theorem \ref{tight} shows that we cannot improve Theorem \ref{main} (2) for the box dimensions of the horizon of a prevalent function in $C_\alpha[0,1]^2$.  However, the horizon property does hold for prevalent functions if we consider the subspace $L_\alpha[0,1]^2$ of $C_\alpha[0,1]^2$ consisting of $\alpha$-Lipschitz functions, that is functions for which
\begin{equation} \label{lip}
\text{Lip}_\alpha(f) = \sup_{\substack{x,y \in [0,1]^2 \\ x \neq y}} \frac{\lvert f(x)-f(y) \rvert}{\, \, \,\,\, \, \, \, \, \lvert x-y\rvert^{3-\alpha}} <\infty.
\end{equation}
It is easily verified that 
\[
\|f\|_{\text{\rm{Lip}}_\alpha} = \| f \|_\infty + \text{\rm{Lip}}_\alpha(f)
\]
defines a complete norm on $L_\alpha$, and we write $d_{\text{\rm{Lip}}_\alpha}$ for the corresponding metric.
\\ \\
The Lipschitz condition controls the box dimension of both the surface and the horizon. Thus if $f \in L_\alpha[0,1]^2$ then $\overline{\dim}_\text{B} G_{f} \leq  \alpha$ (though the converse is not true) 
and 
\begin{equation}\label{horlip}
\overline{\dim}_\text{B} G_{H(f)} \leq  \alpha -1,
\end{equation}\label{horlip}
and this enables the following theorem.

\begin{thm}\label{lipthm}
The set
\[
\{f \in L_\alpha[0,1]^2 : \dim_\text{\emph{B}} G_f = \alpha \text{ and }   \dim_\text{\emph{B}} G_{H(f)} = \alpha - 1 \}
\]
is a 1-prevalent subset of $(L_\alpha[0,1]^2, d_{\text{\emph{Lip}}_\alpha})$.
\end{thm}

Thus a prevalent function in the space of $\alpha$-Lipshitz functions satisfies the horizon property for box dimension.
\\ \\
It would clearly be desirable to obtain analogues of Theorem \ref{main} (2) for \emph{Hausdorff} dimension, $\dim_\text{\H}$.  However, it follows from a category theoretic argument in \cite{graphsums} that the set
\[
H_\alpha[0,1]^2 = \{f \in C[0,1]^2 : \dim_\text{\H} G_f \leq \alpha \}
\]
is not a subspace of $C[0,1]^2$ for $\alpha \in [2, 3)$ because it is not closed under addition (see the remarks at the end of Section 2).  As a consequence, if one were to search for an analogous space to $C_\alpha[0,1]^2$ for Hausdorff dimension one would have to look for a subspace of $C[0,1]^2$ contained in $H_\alpha[0,1]^2$ which would necessarily lie \emph{strictly} inside $H_\alpha[0,1]^2$.

\section{Box dimensions of functions}

The box counting dimensions were defined in (\ref{lbox})--(\ref{ubox}), and in this section we  present various technical results concerning the box dimension of fractal curves and surfaces.  
\\ \\
Let $d \in \mathbb{N}$, let $f \in C[0,1]^d$ and let $S \subseteq [0,1]^d$.  We define the \emph{range} of $f$ on $S$ as
\[
R_f(S) = \sup_{x,y \in S} \lvert f(x)-f(y)\rvert.
\]
For $\delta>0$ let $\Delta_\delta^d$ be the set of grid cubes in the $\delta$-mesh on $[0,1]^d$ defined by
\[
\Delta_\delta^d = \bigcup_{n_1=0}^{\lceil \delta^{-1} \rceil-1} \cdots  \bigcup_{n_d=0}^{\lceil \delta^{-1} \rceil-1} \Big\{[n_1 \delta,(n_1+1)\delta] \times \cdots \times  [n_d \delta,(n_d+1)\delta]\Big\}.
\]
It follows that 
\begin{equation} \label{boxupper}
\delta^{-1}\sum_{S \in \Delta_\delta^d} R_f(S) \leq N_\delta (G_f) \leq 2(\delta^{-1}+1)^d +\delta^{-1}\sum_{S \in \Delta_\delta^d} R_f(S), 
\end{equation}
see \cite{falconer}, so, given a non-constant $f$,
\begin{equation}
N_\delta (G_f) \asymp \delta^{-1}\sum_{S \in \Delta_\delta^d} R_f(S), \label{estimate}
\end{equation}
i.e., there exists constants $\delta_f,C_f>0$ such that for $\delta<\delta_f$
\[
\frac{1}{C_f} \leq \frac{N_\delta (G_f) }{\delta^{-1}\sum_{S \in \Delta_\delta^d} R_f(S)} \leq C_f
\]
The remainder of this section will be devoted to studying the box dimensions of sums of functions.  We will assume throughout that $f+g, f$ and $g$ are all non-constant, as otherwise the proofs are trivial.

\begin{lma} \label{upperbound}
Let $f,g \in C[0,1]^d$.  Then
\[
\overline{\dim}_\text{\emph{B}} G_{f+g} \leq \max\{\overline{\dim}_\text{\emph{B}} G_{f}, \overline{\dim}_\text{\emph{B}} G_{g} \}.
\]
In particular, $C_\alpha[0,1]^d$ is a vector space.
\end{lma}

\begin{proof}
Let $s=\max\{\overline{\dim}_\text{B} G_{f}, \overline{\dim}_\text{B} G_{g} \}$ and let $\epsilon>0$.  By (\ref{estimate}) there exists $\delta_0>0$ such that for all $\delta<\delta_0$ we have
\[
\sum_{S \in \Delta_\delta^d} R_f(S) \leq \delta^{-\overline{\dim}_\text{B} G_{f}-\epsilon+1} \leq \delta^{1-s-\epsilon}
\]
and
\[
\sum_{S \in \Delta_\delta^d} R_g(S) \leq \delta^{-\overline{\dim}_\text{B} G_{g}-\epsilon+1} \leq \delta^{1-s-\epsilon}.
\]
By considering the range of $f+g$ we have
\[
\sum_{S \in \Delta_\delta^d} R_{f+g}(S) \leq \sum_{S \in \Delta_\delta^d} R_{f}(S)+ \sum_{S \in \Delta_\delta^d} R_{g}(S) \leq 2 \delta^{1-s-\epsilon}
\]
so  $\overline{\dim}_\text{B} G_{f+g} \leq s+\epsilon$.  Since this is true for all $\epsilon>0$ we conclude that $\overline{\dim}_\text{B} G_{f+g} \leq s$.
\end{proof}

\begin{lma} \label{upper equals}
Let $f,g \in C[0,1]^d$ and suppose $\overline{\dim}_\text{\emph{B}} G_{f}\neq \overline{\dim}_\text{\emph{B}} G_{g}$.  Then
\[
\overline{\dim}_\text{\emph{B}} G_{f+g} = \max\{\overline{\dim}_\text{\emph{B}} G_{f}, \overline{\dim}_\text{\emph{B}} G_{g} \}.
\]
\end{lma}

\begin{proof}
Let $f,g \in C[0,1]^d$ and suppose, without loss of generality, that $\overline{\dim}_\text{B} G_{f}< \overline{\dim}_\text{B} G_{g}$.  If $f+g = h$ where $\overline{\dim}_\text{B} G_{h} \neq \overline{\dim}_\text{B} G_{g}$, Lemma \ref{upperbound} gives that $\overline{\dim}_\text{B} G_{h} < {\dim}_\text{B} G_{g}$. This contradicts Lemma \ref{upperbound} since
\[
\overline{\dim}_\text{B} G_{h-f}=\overline{\dim}_\text{B} G_{g} > \max\{\overline{\dim}_\text{B} G_{h}, \overline{\dim}_\text{B} G_{-f} \}.
\]
\end{proof}

\begin{lma} \label{but 2}
Let $f,g \in C[0,1]^d$. Then
\[
\overline{\dim}_\text{\emph{B}} G_{f+\lambda g} = \max\{\overline{\dim}_\text{\emph{B}} G_{f}, \overline{\dim}_\text{\emph{B}} G_{g} \}
\]
for all $\lambda \in \mathbb{R}$ with the possible exceptions of $\lambda = 0$ and one other value of $\lambda$.
\end{lma}

\begin{proof}
Let $f, g \in C[0,1]^d$ and assume without loss of generality that
\[
\max\{\overline{\dim}_\text{B} G_{f}, \overline{\dim}_\text{B} G_{g} \} = \overline{\dim}_\text{B} G_{g} = s.
\]
Suppose $\lambda \in \mathbb{R}$ is such that $f+\lambda g = h$ where $\overline{\dim}_\text{B} G_{h} \neq s$.  It follows from Lemma \ref{upperbound} that $\overline{\dim}_\text{B} G_{h} < s$.  Now let $\beta \in \mathbb{R} \setminus \{0,\lambda \}$. Then $f+\beta g = h+(\beta - \lambda) g$ and since $\overline{\dim}_\text{B} G_{h}<\overline{\dim}_\text{B} G_{g}$ we have by Lemma \ref{upper equals} that 
\[
\overline{\dim}_\text{B} G_{f+\beta g}=\max\{\overline{\dim}_\text{B} G_{h}, \overline{\dim}_\text{B} G_{(\beta-\lambda)g} \} = s.
\]
\end{proof}

We write  $\mathcal{L}^1$ for Lebesgue measure on  $\mathbb{R}$.

\begin{lma} \label{leq}
Let $f,g \in C[0,1]^d$.  Then
\[
\underline{\dim}_\text{\emph{B}} G_{f+\lambda g} \geq \max \{ \underline{\dim}_\text{\emph{B}} G_{f}, \underline{\dim}_\text{\emph{B}} G_{g} \}
\]
for $\mathcal{L}^1$-almost all $\lambda \in \mathbb{R}$.
\end{lma}

\begin{proof}
Let $f,g \in C[0,1]^d$, let $\epsilon>0$ and suppose $\max \{ \underline{\dim}_\text{B} G_{f}, \underline{\dim}_\text{B} G_{g} \} =\underline{\dim}_\text{B} G_{g}=s$.  By (\ref{estimate}) there exists a $\delta_0>0$ such that for $\delta<\delta_0$
\begin{equation}
\sum_{S \in \Delta_\delta^d} R_g(S) \geq \delta^{1-s+\epsilon}. \label{g}
\end{equation}

Let $E \subset \mathbb{R}$ be any bounded Lebesgue measurable set and fix $\delta<\delta_0$.  Note that, since
\[
\int_E \lvert a-\lambda b\rvert\, d\lambda \geq \tfrac{1}{4} \,\mathcal{L}^1(E)^2 \,\lvert b \rvert,
\]
for all $a,b \in  \mathbb{R}$,
\begin{eqnarray}
\int_E \sum_{S \in \Delta_\delta^d} R_{f+\lambda g}(S)\, d \lambda &\geq& \sum_{S \in \Delta_\delta^d} \int_E  \lvert R_{f}(S)-\lambda R_{g}(S) \rvert \,d \lambda \nonumber\\ \nonumber\\
&\geq& \sum_{S \in \Delta_\delta^d}\tfrac{1}{4} \,\mathcal{L}^1(E)^2 \,   R_{g}(S) \nonumber \\ \nonumber\\
&=& \tfrac{1}{4} \,\mathcal{L}^1(E)^2 \,  \sum_{S \in \Delta_\delta^d} R_{g}(S) \nonumber \\ \nonumber\\
&\geq& \tfrac{1}{4} \, \mathcal{L}^1(E)^2 \, \delta^{1-s+\epsilon} \label{second}
\end{eqnarray}
using  (\ref{g}).
Let $n \in \mathbb{N}$ and
\[
E_\delta^n = \Big\{\lambda \in [-n,n]:  \sum_{S \in \Delta_\delta^d} R_{f+\lambda g}(S)  \leq \delta^{1-s+2\epsilon} \Big\}.
\]
By (\ref{second}) we have
\[
 \mathcal{L}^1(E_\delta^n) \delta^{1-s+2\epsilon}\,\geq\, \int_{E_\delta^n}\sum_{S \in \Delta_\delta^d} R_{f+\lambda g}(S) \,d \lambda \, \geq \, \tfrac{1}{4} \, \mathcal{L}^1(E_\delta^n)^2 \, \delta^{1-s+\epsilon}
\]
so  $\mathcal{L}^1(E_\delta^n) \leq 4 \delta^\epsilon.$  Choose $K \in \mathbb{N}$ such that $2^{-k}<\delta_0$ for all $k\geq K$.  Then
\[
\sum_{k \geq K} \mathcal{L}^1(E_{2^{-k}}^n) \leq 4 \sum_{k \geq K}  2^{-k\epsilon} < \infty
\]
so by the Borel-Cantelli Lemma,
\[
\mathcal{L}^1 \Big(\bigcap_{M \in \mathbb{N} } \bigcup_{ k \geq M }E_{2^{-k}}^n \Big) =0,
\]
i.e., for $\mathcal{L}^1$-almost all $\lambda \in [-n,n]$, $\lambda \notin E_{2^{-k}}$ for sufficiently large $k$.  It follows that for $\mathcal{L}^1$-almost all $\lambda \in [-n,n]$ there exists $\delta_\lambda>0$ such that for $\delta<\delta_\lambda$
\[
\sum_{S \in \Delta_\delta^d} R_{f+\lambda g}(S)  > \delta^{1-s+2\epsilon}
\]
and hence
\[
\underline{\dim}_\text{B} G_{f+\lambda g} \geq s -2\epsilon.
\]
Since this is true for arbitrarily small $\epsilon$ and since
$\mathbb{R} = \cup_{n \in \mathbb{N}} [-n,n]
$
the result follows.
\end{proof}

\begin{comment}
\begin{lma} \label{lower eq}
Let $f,g \in C[0,1]^d$.  Then
\[
\underline{\dim}_\text{\emph{B}} G_{f+\lambda g} = \max \{ \underline{\dim}_\text{\emph{B}} G_{f}, \underline{\dim}_\text{\emph{B}} G_{g} \} 
\]
for $\mathcal{L}^1$-almost all $\lambda \in \mathbb{R}$.
\end{lma}

The proof I initially came up with for this Lemma was incorrect and after thinking about it for a while I no longer believe that it is true.  We don't need it for our results though... which is good...
\\ \\

\begin{proof}
Let $f,g \in C[0,1]^d$ and let
\begin{eqnarray*}
\Lambda &=& \big\{\lambda \in \mathbb{R} :  \underline{\dim}_\text{B} G_{f+\lambda g} \neq \max \{ \underline{\dim}_\text{B} G_{f}, \underline{\dim}_\text{B} G_{g} \} \big\} \\ \\
&=& \big\{\lambda \in \mathbb{R} :  \underline{\dim}_\text{B} G_{f+\lambda g} < \max \{ \underline{\dim}_\text{B} G_{f}, \underline{\dim}_\text{B} G_{g} \} \big\} \\ \\
&\hspace{1mm}& \qquad \qquad \cup \, \big\{\lambda \in \mathbb{R} :  \underline{\dim}_\text{B} G_{f+\lambda g} > \max \{ \underline{\dim}_\text{B} G_{f}, \underline{\dim}_\text{B} G_{g} \} \big\} \\ \\
&\,=:& \Lambda_- \cup \Lambda_+
\end{eqnarray*}
Writing $f=(f+\lambda g) - \lambda g$ we have
\[
\Lambda_+ \subseteq \big\{\lambda \in \mathbb{R} :  \underline{\dim}_\text{B} G_{f} < \max \{ \underline{\dim}_\text{B} G_{f+\lambda g}, \underline{\dim}_\text{B} G_{-\lambda g} \} \big\}
\]
and it follows from Lemma \ref{leq} that
\[
\mathcal{L}^1 (\Lambda) = \mathcal{L}^1 (\Lambda_-)+\mathcal{L}^1 (\Lambda_+) = 0.
\]
\end{proof}
\end{comment}

Thus we have proved the following theorem.

\begin{thm} \label{sum}
Let $f,g \in C[0,1]^d$.  Then
\[
\max \{ \underline{\dim}_\text{\emph{B}} G_{f}, \underline{\dim}_\text{\emph{B}} G_{g} \} \, \leq \,  \underline{\dim}_\text{\emph{B}} G_{f+\lambda g} \, \leq \,  \overline{\dim}_\text{\emph{B}} G_{f+\lambda g} \, = \, \max \{ \overline{\dim}_\text{\emph{B}} G_{f}, \overline{\dim}_\text{\emph{B}} G_{g} \}
\]
for $\mathcal{L}^1$-almost all $\lambda \in \mathbb{R}$.
\end{thm}

\begin{proof}
This combines Lemma \ref{but 2} and Lemma \ref{leq}.
\end{proof}

It would clearly be desirable to have analogous results for the Hausdorff dimension and packing dimension of the graphs of sums of functions.  However this is not possible for Hausdorff dimension.  In particular, the estimate
\begin{equation} \label{question}
\dim_\text{H} G_{f+g} \leq \max\{\dim_\text{H} G_{f}, \dim_\text{H} G_{g} \}
\end{equation}
fails; indeed every $f \in C[0,1]$ can be written as the sum of two functions with graphs Hausdorff dimension 1.  Mauldin and Williams  \cite{graphsums} showed this by an elegant application of the Baire category theorem.  It is well known that the set $A=\{f \in C[0,1] : \dim_\text{H} G_f = 1 \}$ is co-meagre and thus, for any $f \in C[0,1]$, the set $A \cap (A+f)$ is co-meagre and in particular non-empty.  Hence, we may choose $g = f_1=f_2+f$ where both $f_1,f_2 \in A$.  Thus if $f \notin A$ then $f=f_1-f_2$ satisfies
\[
\dim_\text{H} G_{f} = \dim_\text{H} G_{f_1-f_2} > \max\{\dim_\text{H} G_{f_1}, \dim_\text{H} G_{-f_2} \} = 1.
\]
It was shown in \cite{bairefunctions} that the set of functions with lower box dimension equal to 1 is also co-meagre  so the above argument could be modified to obtain a slightly more general result concerning lower box dimension.  In particular, every $f \in C[0,1]$ has a decomposition $f=f_1+f_2$ where $f_1, f_2 \in C[0,1]$ have lower box dimension equal to 1.  Such a decomposition has been constructed explicitly in \cite{decomp}.
\\ \\
We are unaware if the analogue of (\ref{question}) holds for packing dimension, so we ask:
Is it true that for all $f, g \in C[0,1]$ we have
\begin{equation} \label{packing}
\dim_\text{P} G_{f+g} \leq \max\{\dim_\text{P} G_{f}, \dim_\text{P} G_{g} \}?
\end{equation}

\section{The space $(C_\alpha[0,1]^d, d_{\alpha,d})$} \label{space}

To consider prevalent subsets of 
\[
C_\alpha[0,1]^d = \{f \in C[0,1]^d : \overline{\dim}_\text{B} G_f \leq \alpha \} \qquad \text{ for } \alpha \in [d,d+1]
\]
 we need to show that $C_\alpha[0,1]^d$ is a completely metrizable topological vector space.  It follows from Lemma \ref{upperbound} that $C_\alpha[0,1]^d$ is  a vector space and in this section we will construct a suitable metric.
\\ \\
For $\alpha \geq d$ define
\[
V_\alpha[0,1]^d = \{f \in C[0,1]^d : \| f \|_{\alpha,d} < \infty \}
\]
where
\[
\| f \|_{\alpha,d} = \| f \|_\infty + \sup_{0<\delta \leq 1} \frac{\sum_{S \in \Delta_\delta^d} R_{f}(S) }{\delta^{1-\alpha}}.
\]
It is easy to see that $(V_\alpha[0,1]^d, \| \cdot \|_{\alpha,d})$ is a normed space.  (See \cite{massopust} for the relationship between these spaces and Besov spaces.) 
\begin{lma} \label{comp}
Let $\alpha \in [d,d+1]$.  Then $(V_\alpha[0,1]^d, \| \cdot \|_{\alpha,d})$ is a complete normed space.
\end{lma}

\begin{proof}
Let  $(f_n)_n$ be a Cauchy sequence in $(V_\alpha[0,1]^d, \| \cdot \|_{\alpha,d})$. It follows that $(f_n)_n$ is Cauchy  in $\| \cdot \|_\infty$ and so converges uniformly to some $f \in C[0,1]^d$.
By uniform convergence,
\[
\| f \|_\infty + \sup_{\delta_0 <\delta \leq 1} \frac{\sum_{S \in \Delta_\delta^d} R_{f}(S) }{\delta^{1-\alpha}} \leq \limsup_{n \to \infty}\|f_n\|_{\alpha,d},
\]
for all $0<\delta_0<1$, so 
$f \in V_\alpha[0,1]^d$ with  $\|f\|_{\alpha,d} \leq \limsup_{n \to \infty}\|f_n\|_{\alpha,d}$ . In the same way,  for each $m$, 
we see that  $\|f-f_m\|_{\alpha,d} \leq \limsup_{n \to \infty}\|f_n-f_m\|_{\alpha,d}$, so 
$(f_n)_n$ converges to $f$ in $\|\cdot\|_\alpha$.

\end{proof}

\begin{lma} \label{inter}
Let $(X_k, \| \cdot \|_k)_k$ be a decreasing sequence of complete normed vector spaces.  i.e. for all $k \in \mathbb{N}$ we have $X_k \geq X_{k+1}$ and for $x \in X_{k+1}$ we have $\|x\|_{k+1} \geq \|x\|_k$.  Then
\[
\Big(\bigcap_{k \in \mathbb{N}}  X_k , d \Big)
\]
is a complete metric space where the metric $d$ is defined by
\[
d(x,y) = \sum_{k=1}^\infty \min\big\{ 2^{-k}, \|x-y\|_k\big\}.
\]
\end{lma}

\begin{proof}
It is clear that $d$ is defined for every pair $x,y \in \cap_{k \in \mathbb{N}}  X_k$ and that it is a metric.  To show completeness let $(x_n)_n$ be a Cauchy sequence in $(\cap_{k \in \mathbb{N}}  X_k , d )$.  Then for each $k$, $(x_n)_n$ is Cauchy  in $(X_k , \| \cdot \|_k)$.  Since $(X_k , \| \cdot \|_k)$ is complete there exists $x^{(k)} \in X_k$ such that
$
\|x_n-x^{(k)}\|_k \to 0,
$
but since $\|x\|_k \geq \|x\|_{j}$ for $j<k$ we have that
$
\|x_n-x^{(k)}\|_{j} \to 0
$
for all or $j<k$. Thus $x^{(k)}$ is independent of $k$ and we may simply refer to it as $x$.  It follows that $x \in \cap_{k \in \mathbb{N}} X_k$ and
$
\|x_n-x\|_k \to 0
$
for all $k \in \mathbb{N}$, so $d( x_n, x ) \to 0$.
\end{proof}

Note that whilst the metric $d$ is translation invariant,  $d(0,x)$ does not define a norm as it clearly fails the scalar property.

\begin{lma} \label{dec}
Let $\alpha \geq d$.  For all $k \in \mathbb{N}$ and $f \in C[0,1]^d$ we have
\[
\|f \|_{\alpha+\frac{1}{k+1},d} \geq \|f \|_{\alpha+\frac{1}{k},d}
\]
and consequently
\[
V_{\alpha+\frac{1}{k}}[0,1]^d \geq V_{\alpha+\frac{1}{k+1}}[0,1]^d.
\]
\end{lma}

\begin{proof}
Let $\alpha \geq d$, let $k \in \mathbb{N}$ and let $f \in C[0,1]^d$.  Then
\[
\|f \|_{\alpha+\frac{1}{k},d}  = \| f \|_\infty + \sup_{0<\delta \leq 1} \frac{\sum_{S \in \Delta_\delta^d} R_{f}(S) }{\delta^{1-\alpha-\frac{1}{k}}} \leq  \| f \|_\infty + \sup_{0<\delta \leq 1} \frac{\sum_{S \in \Delta_\delta^d} R_{f}(S) }{\delta^{1-\alpha-\frac{1}{k+1}}}=  \|f \|_{\alpha+\frac{1}{k+1},d}.
\]

\end{proof}

%extra bit maybe for thesis
\begin{comment}
It follows from this and the definition of $V_{\alpha}^2$ that
\[
V_{\alpha+\frac{1}{k}}^2 \geq V_{\alpha+\frac{1}{k+1}}^2.
\]

Combining Lemmas \ref{comp}-\ref{dec} we have that
\[
\Big(\bigcap_{k \in \mathbb{N}}  V_{\alpha+\frac{1}{k}}[0,1]^d , d_{\alpha,d} \Big)
\]
is a complete metric space where the metric $d_{\alpha,d}$ is given by
\[
d_{\alpha,d}(f,g) = \sum_{k=1}^\infty \min\big\{ 2^{-k}, \|f-g\|_{\alpha+\frac{1}{k}}\big\}.
\]
\end{comment}

\begin{prop} \label{met}
Let $\alpha \in [d,d+1)$. Then
\[
 \{f \in C[0,1]^d : \overline{\dim}_\text{\emph{B}} G_{f} \leq \alpha\} \equiv C_\alpha[0,1]^d = \bigcap_{k \in \mathbb{N}}  V_{\alpha+\frac{1}{k}}[0,1]^d.
\]
Moreover $(C_\alpha[0,1]^d,d_{\alpha,d} )$ is a complete metric space where 
\[
d_{\alpha,d}(f,g) = \sum_{k=1}^\infty \min\big\{ 2^{-k}, \|f-g\|_{\alpha+\frac{1}{k},d}\big\}.
\]
\end{prop}

\begin{proof}
It follows from Lemmas \ref{comp}-\ref{dec}  that
$\big(\bigcap_{k \in \mathbb{N}}  V_{\alpha+\frac{1}{k}}[0,1]^d , d_{\alpha,d} \big)$ is a complete metric space.
\\ \\
 Let $f \in C_\alpha[0,1]^d$ so that $\overline{\dim}_\text{B}G_f \leq \alpha$, and so by (\ref{estimate}), for each $k \in \mathbb{N}$, there exists $\delta_0>0$ such that for $\delta<\delta_0$
\[
\delta^{-1} \sum_{S \in \Delta_\delta^d} R_{f}(S) \leq N_\delta(G_f) \leq \delta^{-\alpha-\frac{1}{k}}.
\]
It follows that
$\|f \|_{\alpha+\frac{1}{k},d} < \infty$
so $f \in V_{\alpha+\frac{1}{k}}[0,1]^d$ for each  $k$.

For the opposite inclusion,   let $f \in \bigcap_{k \in \mathbb{N}}  V_{\alpha+\frac{1}{k}}[0,1]^d$.  Hence, for all $k \in \mathbb{N}$,
\[
\sup_{0<\delta \leq 1} \frac{\sum_{S \in \Delta_\delta^d} R_{f}(S) }{\delta^{1-\alpha-\frac{1}{k}}}< \infty,
\]
and hence for all $\delta \in (0,1]$
\[
\delta^{-1} \sum_{S \in \Delta_\delta^d} R_{f}(S) \leq C_k \delta^{-\alpha-\frac{1}{k}}
\]
where $C_k$ depends only on $f$ and $k$.  Taking logarithms and using (\ref{estimate}),
$\overline{\dim}_\text{B} G_f \leq \alpha+\frac{1}{k}$ for all $k$, so  $\overline{\dim}_\text{B} G_f \leq \alpha$.

\end{proof}

It is easy to see that the vector space operations are continuous with respect to the topology induced by $d_{\alpha,d}$ and therefore $(C_\alpha[0,1]^d, d_{\alpha,d})$ is a completely metrizable topological vector space. Thus we are able to consider prevalent subsets of $C_\alpha[0,1]^d$.

%Note that for $\alpha \in [1,2)$ we can construct the space $(C_\alpha[0,1], d_
%{\alpha,1})$ in a similar way.

\section{Proofs of Theorems  \ref{01}, \ref{main} and   \ref{lipthm}} \label{proof1}  

In this section we will prove Theorem \ref{main} which provides bounds for the box dimensions of the horizon of a prevalent surface in $C_\alpha[0,1]^2$ and also
the Lipschitz variant, Theorem   \ref{lipthm}.  We will also sketch the proof of Theorem \ref{01}, which is very similar but simpler to the proof of Theorem \ref{main}.  We begin with a technical measurability lemma.
\begin{lma} \label{bor}\hspace{1mm}

\begin{itemize}
\item[(1)] For all $\alpha \in [2,3)$, the set $F_\alpha[0,1]^2$ is a Borel subset of $(C_\alpha[0,1]^2, d_{\alpha,2})$;
\item[(2)] $F_3[0,1]^2$ is a Borel subset of $(C[0,1]^2, d_\infty)$.
\end{itemize}
\end{lma}
\vspace{2mm}
\begin{proof}  We will prove part (1); part (2) is similar.
\\ \\
Let $\alpha \in [2,3)$.  We have
\begin{eqnarray*}
F_\alpha[0,1]^2 
%&=& \{f \in C[0,1]^2 : \dim_\text{B} G_f = \alpha \text{ and } \alpha-1 \leq 
%\underline{\dim}_\text{B} G_H(f) \leq  \overline{\dim}_\text{B} G_H(f) \leq 2\} \\ \\
&=& \{f \in C[0,1]^2 : \dim_\text{B} G_f = \alpha\} \cap \{f \in C[0,1]^2 : \alpha-1 \leq \underline{\dim}_\text{B} G_H(f) \leq  \overline{\dim}_\text{B} G_H(f) \leq 2\} \\ \\
&\equiv& F_{1} \cap F_2 ,
\end{eqnarray*}
say.
We will first consider $F_1$.  By (\ref{estimate}) we have
\begin{eqnarray*}
F_1 &=& \bigcap_{\substack{q \in \mathbb{Q}\\ q>0}} \bigcup_{N \in \mathbb{N}} \bigcap_{n \geq N} \Bigg\{ f \in C[0,1]^2 : \,\,\,\,2^{(\alpha-q-1)n} < \sum_{\,\,\,\,\,S \in \Delta_{2^{-n}}^2} R_f(S) \,\,\,\,< 2^{(\alpha+q-1)n}\Bigg\}
\end{eqnarray*}
and it is clear that the set in curly brackets  is open in $d_\infty$ for all $q, N$ and $n$. Since the metric $d_{\alpha,2}$ is stronger than $d_\infty$, this set is also open in   $d_{\alpha,2}$,
so $F_1$ is a Borel subset of $(C_\alpha[0,1]^2, d_{\alpha,2})$.
\\ \\
Similarly for $F_2$,
\begin{eqnarray*}
F_2 &=& \bigcap_{\substack{q \in \mathbb{Q}\\ q>0}} \bigcup_{N \in \mathbb{N}} \bigcap_{n \geq N}  \Bigg\{ f \in C[0,1]^2 : \,\,\,\,2^{(\alpha-1-q-1)n} < \sum_{\,\,\,\,\,S \in \Delta_{2^{-n}}^1} R_{H(f)}(S) \,\,\,\,< 2^{(2+q-1)n}\Bigg\}.
\end{eqnarray*}
If $g \in B_{d_\infty}(f,r)$ then $H(g) \in B_{d_\infty}(H(f),r)$ so the set in curly brackets is open in $d_\infty$ and thus in $d_{\alpha,2}$.  Hence $F_2$ is a Borel subset of $(C_\alpha[0,1]^2, d_{\alpha,2})$, and consequently $F_{\alpha}[0,1]^2$ is Borel.
\end{proof}

We will now turn to the proof of Theorem \ref{main}.  We will begin by constructing a probe.  Let $\alpha \in [2,3]$ and let $\psi_\alpha \in C[0,1]$ be a non-constant function satisfying
\[
\underline{\dim}_\text{B} G_{\psi_\alpha} =\overline{\dim}_\text{B} G_{\psi_\alpha}= \alpha-1.
\]
Define $\Psi_\alpha \in C[0,1]^2$ by
\[
\Psi_\alpha(x,y)=\psi_\alpha(x)
\]
for $(x,y) \in [0,1]^2$.  It is clear that 
\[
\underline{\dim}_\text{B} G_{\Psi_\alpha} =\overline{\dim}_\text{B} G_{\Psi_\alpha} = \dim_\text{B} G_{\psi_\alpha} +1 = \alpha
\]
and hence $\Psi_\alpha \in C_\alpha[0,1]^2$.  Also note that $H(\Psi_\alpha)=\psi_\alpha$.

\begin{lma}  \label{key} Let $\alpha \in [2,3]$ and let $f \in C_\alpha[0,1]^2$.  For $\mathcal{L}^1$-almost all $\lambda \in \mathbb{R}$ we have
\begin{itemize}
\item[(1)] $\dim_\text{\emph{B}} G_{f+\lambda \Psi_\alpha} = \dim_\text{\emph{B}} G_{\Psi_\alpha} = \alpha$;
\item[(2)] $\alpha - 1  \leq  \underline{\dim}_\text{\emph{B}} G_{H(f+\lambda \Psi_\alpha)} \leq \overline{\dim}_\text{\emph{B}} G_{H(f+\lambda \Psi_\alpha)} \leq 2$.
\end{itemize}
\end{lma}

\begin{proof} Let $\alpha \in [2,3]$ and let $f \in C_\alpha[0,1]^2$.\\ \\
(1)  This follows immediately from Theorem \ref{sum}.
\\ \\
(2)  Since $\Psi_\alpha(x,y)$ is independent of $y$,
\begin{eqnarray*}
H(f+\lambda \Psi_\alpha)(x) = \sup_{y \in [0,1]} (f+\lambda \Psi_\alpha)(x,y) &=& \sup_{y \in [0,1]} \Big(f(x,y)+\lambda \Psi_\alpha(x,y) \Big) \\ \\
&=&\sup_{y \in [0,1]} \Big(f(x,y)\Big)+\lambda H(\Psi_\alpha)(x) \\ \\
&=&H(f)(x)+\lambda H(\Psi_\alpha)(x). 
\end{eqnarray*}
Applying Theorem \ref{sum} for $H(f), H(\Psi_\alpha) \in C[0,1]$ gives
\[
\alpha - 1  =  \underline{\dim}_\text{B} G_{H(\Psi_\alpha)} \leq  \underline{\dim}_\text{B} G_{H(f+\lambda \Psi_\alpha)} \leq \overline{\dim}_\text{B} G_{H(f+\lambda \Psi_\alpha)} \leq 2
\]
for $\mathcal{L}^1$-almost all $\lambda \in \mathbb{R}$.
\end{proof}

Let $\alpha \in [2,3]$ and let $P_\alpha =\{\lambda \Psi_\alpha : \lambda \in \mathbb{R} \}\subset C[0,1]^2$.  Define $\pi_\alpha: P_\alpha \to \mathbb{R}$ by
\[
\pi_\alpha(\lambda \Psi_\alpha) = \lambda
\]
and a measure $\mathcal{L}_{P_\alpha}$ on $P_\alpha$ by
\[
\mathcal{L}_{P_\alpha} = \mathcal{L}^1 \circ \pi_\alpha.
\]

\begin{lma}\label{probe1}
 $P_\alpha$ is a probe for $F_\alpha[0,1]^2$.  In particular, for all $f \in C_\alpha[0,1]^2$, 
\[
\mathcal{L}_{P_\alpha} \Big( C_\alpha[0,1]^2 \setminus (F_\alpha[0,1]^2+f) \Big)=0.
\]
\end{lma}

\begin{proof}

Let $f \in C_\alpha[0,1]^2$.  Then
\begin{eqnarray*}
\mathcal{L}_{P_\alpha} \Big( C_\alpha[0,1]^2 \setminus (F_\alpha[0,1]^2+f) \Big) &=& \mathcal{L}_{P_\alpha} \Big(P_\alpha \setminus (F_\alpha[0,1]^2+f) \Big) \\ \\
&=& (\mathcal{L}^1 \circ \pi_\alpha ) \Big(\lambda \Psi_\alpha \in P_\alpha : \lambda \Psi_\alpha - f \notin F_\alpha[0,1]^2 \Big) \\ \\
&=& \mathcal{L}^1  \Big(\lambda \in \mathbb{R} : \lambda \Psi_\alpha - f \notin F_\alpha[0,1]^2 \Big) \\ \\
&=& 0
\end{eqnarray*}
by Lemma \ref{key}.
\end{proof}

Lemma \ref{probe1} combined with Lemma \ref{bor}(1-2) gives Theorem \ref{main}.

\begin{proof}[Proof of Theorem \ref{lipthm}]\quad 
\\
This is very similar to the proof of Theorem  \ref{main}. We note that for $f \in L_\alpha[0,1]^2$, inequality
(\ref{horlip}) allows us to tighten Lemma \ref{key} to the following statement: For $\alpha \in [2,3]$ and $f \in L_\alpha[0,1]^2$, 
$$\underline{\dim}_\text{B} G_{H(f+\lambda \Psi_\alpha)}= \overline{\dim}_\text{B} G_{H(f+\lambda \Psi_\alpha)} =\alpha -1$$
for $\mathcal{L}^1$-almost all $\lambda \in \mathbb{R}$. 
With this equality, the remainder of the proof is virtually identical to that of Theorem  \ref{main}.
\end{proof}

\begin{proof}[Proof of Theorem \ref{01}]\quad 
\\
Again, this is very similar to the proof of Theorem \ref{main}.  For $\alpha \in [d,d+1]$ a one dimensional probe is constructed for $D_\alpha[0,1]^d$ using a function $\Psi_\alpha\in C_\alpha[0,1]^d$ with $\dim_\text{B} G_{\Psi_\alpha} = \alpha$.  It then follows from Theorem \ref{sum} that for all $f \in C_\alpha[0,1]^d$,
\[
\dim_\text{B} G_{f+\lambda \Psi_\alpha} = \dim_\text{B} G_{\Psi_\alpha} = \alpha
\]
for $\mathcal{L}^1$-almost all $\lambda \in \mathbb{R}$.  The rest of the proof is straightforward.
\end{proof}

\section{Proof of Theorem \ref{tight} } \label{proof2}

In this section we will prove Theorem \ref{tight} which shows that the bounds obtained in 
Theorem \ref{main} are as sharp as possible.
\\ \\
The key tools will be certain classes of functions in $C[0,1]^2$ which we will call `forcers' and `modifiers'.  We will first show such functions exist by explicit construction.  We prove, given a compact subset $K \subset C_\alpha[0,1]^2$ and a point $y_0 \in [0,1]$, the existence of a `forcer', $F_{K,y_0} \in C_2[0,1]^2$, which `forces' the horizon of $F_{K,y_0}+f$ to lie along the line $[0,1] \times {y_0}$ for all $f \in K$.  Furthermore, we prove, given $g \in C[0,1]$ and a point $y_0 \in [0,1]$, the existence of a `modifier', $M_{g,y_0} \in C_2[0,1]^2$, such that
\[
H(F_{K,y_0}+M_{g,y_0}) = g.
\]

\begin{lma}[Forcers]
Let $K \subset C_\alpha[0,1]^2$ be a compact set and let $y_0 \in [0,1]$.  There exists $F_{K,y_0} \in C_2[0,1]^2$ such that
\begin{itemize}
\item[(1)] For all $f \in K$
\[
f(x,y_0) + F_{K,y_0}(x,y_0) \geq f(x,y) + F_{K,y_0}(x,y) 
\]
for all $y \in[0,1]$;

\item[(2)] $F_{K,y_0}(x,y_0)=0$ for all $x \in [0,1]$.
\end{itemize}

\end{lma}

\begin{proof}

Define a map 
$\phi:(K,d_{\alpha,2})  \to (C[0,1]^2,\| \cdot \|_\infty)$ by
\[
\phi(f)(x,y) =f(x,y) - f(x,y_0).
\]

It is clear that $\phi$ is continuous and thus $\phi(K)$ is compact. In particular, $g$ defined by $g(x,y)=\sup_{f \in \phi(K)} f(x,y)$ is continuous on $[0,1]^2$, from which it follows that $g^*$ defined by $g^*(y) = \sup_{x \in [0,1]} g(x,y)$ is continuous. Now let $F \in C[0,1]$ be such that
\begin{itemize}

\item[(i)] $\overline{\dim}_\text{B} G_F = 1$;

\item[(ii)]  $F(y) \geq g^*(y)$ for all $y \in [0,1]$;

\item[(iii)]  $F(y_0) = g^*(y_0) = 0$;

\end{itemize}

and define the `forcer' $F_{K, y_0} \in C_2[0,1]^2$ by
\[
F_{K, y_0} (x,y) = -F(y).
\]
It is clear that a function $F$ satisfying (i)--(iii) exists,  
for example,
$$F(y) = \left\{
\begin{array}{ll}
\max_{y \leq w  \leq y_0} g^*(w) & (0\leq y \leq y_0)      \\
\max_{y_0 \leq w  \leq y} g^*(w)  &   (y_0 \leq y \leq 1)  
\end{array}
\right. .
$$
Note that this function is monotonic on either side of $y_0$, and the box dimension of the graph of a monotonic function is 1.
\\ \\
Clearly (i) implies that $F_{K, y_0} \in C_2[0,1]^2$ and (iii) implies that $F_{K, y_0}$ satisfies (2).  
\\ \\
To check (1), let $f \in K$ and let $x,y \in [0,1]$.  Then
$$
f(x,y) - f(x,y_0) \leq g(x,y)
\leq  F(y)
= -F_{K, y_0} (x,y)  
= F_{K, y_0}(x,y_0) - F_{K, y_0} (x,y) 
$$
as required.

\end{proof}

We now construct `modifiers'; the crucial thing is showing we can construct a surface with dimension equal to 2, the least value possible, but with a given horizon which may be highly irregular.

\begin{lma}[Modifiers]
Let $g \in C[0,1]$ and let $y_0 \in [0,1]$.  Then there exists $M_{g,y_0} \in C[0,1]^2$ such that
\begin{itemize}
\item[(1)] $M_{g,y_0}(x, y_0) = g(x)$ for all $x \in[0,1]$;
\item[(2)] $M_{g,y_0}(x, y) \leq M_{g,y_0}(x, y_0)$ for all $y \in [0,1]$;
\item[(3)] $\dim_{\text{\emph{B}}} G_{M_{g,y_0}} = 2$.
\end{itemize}

In particular, for all $g \in C[0,1]$ and all $y_0 \in [0,1]$ we have that $M_{g,y_0} \in C_\alpha[0,1]^2$ for all $\alpha \in [2,3].$
\end{lma}

\begin{proof}
Let $g \in C[0,1]$ and without loss of generality assume that $y_0 = 0$; the proof can easily be adapted for arbitrary $y_0$.\\ \\
Let $(p_k)_k$ be a sequence of polynomials such that
\begin{itemize}
\item[(i)]  $p_1 \leq p_2 \leq p_3 \leq \dots$;
\item[(ii)] $p_k \nearrow g$;
\item[(iii)] $\lvert p_k(x_1) - p_k(x_2) \rvert \leq 2^k \lvert x_1- x_2 \rvert$  for all $k$ and for all $x_1, x_2 \in [0,1]$.
\end{itemize}

To achieve this, the Weierstrass approximation theorem allows us to chose a sequence satisfying (i) and (ii), and then we may adapt this sequence to get  a very slowly converging sequence, which may involve repeating terms for longer and longer runs, which satisfies (3).
\\ \\
Let 
\[
D_k = [0,1] \times (2^{-k}, 2^{-k+1})
\]
so that we have the decomposition
$
[0,1]^2 = \bigcup_{k=1}^{\infty} \overline{D_k},
$
and let $q:[0,1] \to [0,1]$ be defined by
\[
q(x) = \left\{ \begin{array}{c l}
0  & x =0 \\
2^k x -1   & x \in [2^{-k}, 2^{-k+1}) \\
1 &  x=1
\end{array}.
\right. 
\]
Now define $M_{g,0}: [0,1]^2 \to \mathbb{R}$ by
\[
M_{g,0}(x,y) = \left\{ \begin{array}{c l}
g(x) & y=0 \\ 
q(y) p_{k-1}(x) + (1-q(y))p_{k}(x) &  y \in [2^{-k}, 2^{-k+1}) \text{ for some } k \in \mathbb{N}
\end{array}
\right. 
\]
It is clear that $M_{g,0}$ satisfies conditions (1) and (2) of being a `modifier'.  It remains to show that $\dim_{\text{B}} G_{M_{g,0}} = 2$.
\\ \\
Let $n \in \mathbb{N}$.  By (\ref{boxupper}) we have
\begin{eqnarray}
N_{2^{-n}}(G_{M_{g,0}}) &\leq& 2(2^n+1)^2 +2^{n}\sum_{S \in \Delta_{2^{-n}}^2} R_{M_{g,0}}(S) \nonumber \\ \nonumber \\
&=& 2(2^n+1)^2+2^{n}\sum_{k=1}^{n} \quad \sum_{\substack{S \in \Delta_{2^{-n}}^2\\ \\
S \cap D_k \neq \emptyset}} R_{M_{g,0}}(S)  \quad + \quad 2^{n} \sum_{\substack{S \in \Delta_{2^{-n}}^2\\ \\
\forall k = 1,\dots, n \,: \,S \cap D_k = \emptyset}} R_{M_{g,0}}(S) \nonumber \\ \nonumber \\
&\leq& 2(2^n+1)^2+2^{n}\sum_{k=1}^{n} \quad \sum_{\substack{S \in \Delta_{2^{-n}}^2\\ \\
S \cap D_k \neq \emptyset}} R_{M_{g,0}}(S)  \quad + \quad 2^{n} (2^n+1) 2\|M_{g,0}\|_\infty \nonumber \\ \nonumber \\
&=& 2^{n}\sum_{k=1}^{n} \quad \sum_{\substack{S \in \Delta_{2^{-n}}^2\\ \\
S \cap D_k \neq \emptyset}} R_{M_{g,0}}(S)  \quad + \quad O \Big( (2^{n})^2 \Big) \label{rangeest}
\end{eqnarray}
Let $S=S_x \times S_y \in \Delta_{2^{-n}}^2$ be such that $S \cap D_k \neq \emptyset$.  Then

\begin{eqnarray*}
R_{M_{g,0}}(S) &\leq& R_{q p_{k-1}}(S) + R_{(1-q)p_{k}}(S) \\ \\
&\leq& \Bigg( \sup_{y \in S_y} \sup_{x_1,x_2 \in S_x} \Big\lvert q(y) \big(p_{k-1}(x_1) - p_{k-1}(x_2) \big) \Big\rvert  + \sup_{x \in S_x}  \sup_{y_1,y_2 \in S_y} \Big\lvert p_{k-1}(x) \big(q(y_1) - q(y_2) \big)\Big\rvert \Bigg) \\ \\
&\quad& \quad + \Bigg( \sup_{y \in S_y} \sup_{x_1,x_2 \in S_x} \Big\lvert (1-q(y)) \big(p_{k}(x_1) - p_{k}(x_2) \big) \Big\rvert  \\ \\
&\qquad&  \qquad \qquad \qquad \qquad \qquad + \sup_{x \in S_x}  \sup_{y_1,y_2 \in S_y} \Big\lvert p_{k}(x) \big((1-q(y_1)) -(1- q(y_2)) \big)\Big\rvert \Bigg) \\ \\
&=& \sup_{y \in S_y} q(y) \Bigg( \sup_{x_1,x_2 \in S_x} \Big\lvert  p_{k-1}(x_1) - p_{k-1}(x_2) \Big\rvert \Bigg) + \sup_{x \in S_x}  p_{k-1}(x) \Bigg(\sup_{y_1,y_2 \in S_y} \Big\lvert q(y_1) - q(y_2) \Big\rvert \Bigg)  \\ \\
&\quad& \quad + \sup_{y \in S_y} (1-q(y))  \Bigg( \sup_{x_1,x_2 \in S_x} \Big\lvert  p_{k}(x_1) - p_{k}(x_2) \Big\rvert \Bigg)  +\sup_{x \in S_x}  p_{k}(x)  \Bigg(\sup_{y_1,y_2 \in S_y} \Big\lvert  q(y_2)- q(y_2) \Big\rvert \Bigg) \\ \\
&=&  2^{k-1-n} +  \|g\|_\infty2^{k-n} +  2^{k-n} + \|g\|_\infty 2^{k-n}\\ \\
&=& c \, 2^{k-n}
\end{eqnarray*}
where $c = \frac{3}{2}+2\|g\|_\infty$.  Combining this with (\ref{rangeest}) we obtain
\begin{eqnarray*}
N_{2^{-n}}(G_{M_{g,0}}) &\leq& 2^{n}\sum_{k=1}^{n} \Big\lvert \Big\{ S \in \Delta_{2^{-n}}^2 : S \cap D_k \neq \emptyset \Big\} \Big\rvert \, c \, 2^{k-n} \quad + \quad O \Big( (2^{n})^2 \Big) \\ \\
&=& 2^{n}\sum_{k=1}^{n} 2^n \, 2^{n-k} \, c \, 2^{k-n} \quad + \quad O \Big( (2^{n})^2 \Big) \\ \\
&=& c \, n \, (2^{n})^2  \quad + \quad O \Big( (2^{n})^2 \Big)
\end{eqnarray*}
and letting $n \to \infty$ we deduce that
$
\dim_{\text{B}} G_{M_{g,0}} = 2.
$
\end{proof}

Forcers and modifiers will be used throughout the remainder of this section without being mentioned explicitly.  We may now complete the proof of Theorem \ref{tight}(1).
\\
\begin{proof}[Proof of Theorem \ref{tight} (1)]\quad 

Let $\alpha \in [2,3)$ and assume $U_\alpha[0,1]^2$ is a prevalent subset of $(C_\alpha[0,1]^2,d_{\alpha,2})$.  Hence, there exists a Borel measure $\mu_1$ on $C_\alpha[0,1]^2$ and a compact set $K_1 \subset C_\alpha[0,1]^2$ such that

\begin{equation*}
0 < \mu_1(K_1)<\infty
\end{equation*}
and
\begin{equation} \label{mu1}
\mu_1 \Big(C_\alpha[0,1]^2 \setminus (U_\alpha[0,1]^2+f) \Big) = 0 \text{ for all } f \in C_\alpha[0,1]^2. 
\end{equation}

Let $f_1, f_2 \in C[0,1]$ be functions such that
\begin{equation} \label{opposite}
\overline{\dim}_\text{B}\, G_{ f_1} = 2
\end{equation}
and 
\begin{equation} \label{opposite2}
\overline{\dim}_\text{B}\, G_{ f_2} = 1.
\end{equation}

It follows from (\ref{mu1}) that for $\mu_1$-almost all $f \in C_\alpha[0,1]^2$ we have
\[
f \in \Big(U_\alpha[0,1]^2 - (F_{K_1, 0}+M_{f_1, 0}) \Big) \cap \Big(U_\alpha[0,1]^2 - (F_{K_1, 0}+M_{f_2, 0}) \Big)
\]
and since $\mu_1(K_1)>0$, we can choose $f_0 \in K_1$ such that
\[
f_0 \in \Big(U_\alpha[0,1]^2 - (F_{K_1, 0}+M_{f_1, 0}) \Big) \cap \Big(U_\alpha[0,1]^2 - (F_{K_1, 0}+M_{f_2, 0}) \Big).
\]
Hence, there exist $h_1, h_2 \in U_\alpha[0,1]^2$ such that
\[
f_0 = h_1-(F_{K_1, 0}+M_{f_1,0}) = h_2 - (F_{K_1, 0}+M_{f_2,0}).
\]
It follows that
\[
f_0 + (F_{K_1, 0}+M_{f_1,0}) = h_1 \in U_\alpha[0,1]^2
\]
and
\[
f_0 + (F_{K_1, 0}+M_{f_2,0}) = h_2 \in U_\alpha[0,1]^2.
\]
Let $f_0^* \in C[0,1]$ be defined by $f_0^*(x) = f_0(x,0)$.  We will now consider the horizons of $f_0 + (F_{K_1, 0}+M_{f_1,0})$ and $f_0 + (F_{K_1, 0}+M_{f_2,0})$.  We have
\begin{eqnarray*}
H\Big(f_0 + (F_{K_1, 0}+M_{f_1,0})\Big)(x) &=& \sup_{y \in [0,1]} \Big(f_0(x,y) + (F_{K_1, 0}+M_{f_1,0})(x,y) \Big) \\ \\
&=& f_0(x,0)+(F_{K_1, 0}+M_{f_1,0})(x,0) \qquad \qquad \text{by Lemmas 5.1--5.2}\\ \\
&=& f_0^*(x) + f_1(x)
\end{eqnarray*}

and similarly
\begin{eqnarray*}
H\Big(f_0 +(F_{K_1, 0}+M_{f_2,0})\Big)(x) &=&  f_0^*(x) + f_2(x).
\end{eqnarray*}

Since $f_0 + (F_{K_1, 0}+M_{f_1,0})  \in U_\alpha[0,1]^2$ and $f_0 + (F_{K_1, 0}+M_{f_2,0})  \in U_\alpha[0,1]^2$, it follows that
\begin{equation} \label{bothlessthan2}
\overline{\dim}_\text{B}\, G_{ f_0^*  + f_1}, \, \, \overline{\dim}_\text{B} \,G_{ f_0^* + f_2}  < 2.
\end{equation}
Since $\overline{\dim}_\text{B} G_{f_1} = 2$ it follows from Lemma \ref{upper equals} that $\overline{\dim}_\text{B} G_{f_0^*} = 2$.  It now follows from (\ref{opposite2}) and Lemma \ref{upper equals} that $\overline{\dim}_\text{B} G_{f_0^* + f_2} = 2$ which contradicts (\ref{bothlessthan2}).
\end{proof}
\medskip

We turn to the proof of Theorem  \ref{tight}(2).

\begin{proof}[Proof of Theorem \ref{tight} (2)]\hspace{1mm}

Let $\alpha \in [2,3)$ and assume that $L_\alpha[0,1]^2$ is a prevalent subset of $(C_\alpha[0,1]^2, d_{\alpha,2})$.  Hence, there exists a Borel measure $\mu_2$ on $C_\alpha[0,1]^2$ and a compact set $K_2 \subset C_\alpha[0,1]^2$ such that
$0 < \mu_2(K_2)<\infty$ and
\begin{equation} \label{mu2}
\mu_2\Big(C_\alpha[0,1]^2 \setminus (L_\alpha[0,1]^2+f) \Big) = 0 \text{ for all } f \in C_\alpha[0,1]^2.
\end{equation}

 Let
\[
V=\{(x,0,z): x, z \in \mathbb{R} \} \in G_{3,2}
\]
and for $a \in \mathbb{R}$ let
\[
V_a=V+(0,a,0) =\{(x,a,z): x, z \in \mathbb{R} \}.
\]
\\ 
We claim that for all $y \in [0,1]$, $\mu_2$-almost all $f \in K_2$ satisfy
\begin{equation}
\overline{\dim}_{\text{B}} \, G_f \cap V_y >  \alpha-1. \label{bad}
\end{equation}
To see this, assume there exists $y_0 \in [0,1]$ such that
\[
\mu_2 \Big(f \in K_2 : \overline{\dim}_{\text{B}} G_f \cap V_{y_0} \leq  \alpha-1 \Big) >0;
\]
then
\[
\mu_2 \Big(f \in K_2 : \overline{\dim}_{\text{B}} G_{H(f + F_{K_2,y_0})} \leq  \alpha-1 \Big) >0
\]
whence
\[
\mu_2 \Big(C_\alpha[0,1]^2 \setminus (L_\alpha[0,1]^2 - F_{K_2, y_0}) \Big) >0,
\]
which contradicts (\ref{mu2}).
\\ \\
We also note that, by results in \cite{boxslice}, if
 $f \in C_\alpha[0,1]^2$, then for $\mathcal{L}^1$-almost all $y \in [0,1]$ we have
\begin{equation}
\overline{\dim}_{\text{B}} \, G_{f} \cap V_y \leq  \alpha-1.\label{slice}
\end{equation}
\\
Define a map $\Phi : K_2 \times [0,1] \to [1,2]$ by
\[
\Phi(f,y) = \overline{\dim}_\text{B} \, G_f \cap V_y,
\]
and observe that $\Phi$ is a measurable function.  To see this let $\Phi_1: K_2 \times [0,1] \to C[0,1]$ and $\Phi_2: C[0,1] \to [1,2]$ be
\[
(\Phi_1 (f,y)) (x) = f(x,y)
\]
and
\[
\Phi_2 (f) = \overline{\dim}_\text{B} G_f.
\]
It is clear that $\Phi_1$ is continuous with respect to the metric $d$ on $K_2 \times [0,1]$ given by
\[
d \Big( (f_1, y_1), (f_2,y_2) \Big) = d_{\alpha,2}(f_1,f_2) +\lvert y_1-y_2\rvert
\]
and, since $\mu$ is a Borel measure, it follows that $\Phi_1$ is measurable with respect to any Borel measure on $(C[0,1], \| \cdot \|_\infty)$.  The function $\Phi_2$ is an upper limit of measurable functions and therefore is itself (Borel) measurable.
It follows that the composition $\Phi = \Phi_2 \circ \Phi_1$ is (Borel) measurable.
\\ \\
Consider the following integral.
\begin{eqnarray*}
\int_{K_2} \int_0^1 \Phi(f,y) \, d\mathcal{L}^1 \, d\mu_2 &=&  \int_0^1 \int_{K_2} \Phi(f,y) \, d\mu_2 \, d\mathcal{L}^1 \qquad \qquad \text{by Fubini's Theorem} \\ \\
&>& \int_0^1 (\alpha - 1) \, \mu_2(K_2) \, d\mathcal{L}^1 \qquad \qquad \text{by (\ref{bad})} \\ \\
&=&  (\alpha - 1) \, \mu_2(K_2).
\end{eqnarray*}

It follows that
\[
\mu_2 \Big(f \in K_2 : \int_0^1 \Phi(f,y) \, d\mathcal{L}^1 > \alpha - 1 \Big) > 0
\]
and we can thus choose $f_0 \in K_2$ such that
\[
\int_0^1 \Phi(f_0,y) \, d\mathcal{L}^1 > \alpha - 1.
\]
It follows that
\[
\mathcal{L}^1 \Big(y \in [0,1] : \Phi(f_0,y)  > \alpha - 1 \Big) > 0
\]
but this contradicts (\ref{slice}).
\end{proof}

\begin{centering}

\textbf{Acknowledgements}

\end{centering}

We thank James T. Hyde for suggesting the Fubini type argument used in completing the proof Theorem \ref{tight} (2).

\end{document}